\documentclass[10pt,a4paper]{amsart}
\usepackage{amsmath,amsthm,amstext,amssymb,a4,stmaryrd, hyperref}
\usepackage{graphics,epsfig}
\usepackage{psfrag}
\usepackage{graphicx}
\usepackage[T1]{fontenc}

\newcommand{\R}{\ensuremath{\mathbb{R}}}

\newcommand{\D}{\ensuremath{\mathrm{d}}}
\newcommand{\tr}{\ensuremath{\mathrm{trace}}}
\newcommand{\del}{{\partial}}

\theoremstyle{plain}
\newtheorem*{Thm}{Theorem}

\theoremstyle{remark}

\theoremstyle{definition}

\begin{document}

\author{A.~Dirmeier$^\dag$}
\address{$^\dag$MINT-Kolleg, Universit\"at Stuttgart, Azenbergstr.~12, 70174 Stuttgart, Germany}
\email{dirmeier@mint.uni-stuttgart.de}

\keywords{Twisted Product Spacetimes, Constant Mean Curvature Foliations}

\subjclass[2010]{Primary 53B30, 53A10; Secondary 53C12}

\title[]{Constant Mean Curvature Foliations of $3$-Dimensional Twisted Product Spacetimes}

\begin{abstract}
For $2+1$ spacetime dimensions, we derive sufficient conditions for the twisting function in a twisted product spacetime, such that there is a global foliation by spacelike CMC surfaces.
\end{abstract}

\maketitle
\section{Introduction}
Foliations of spacetimes by hypersurfaces of constant mean curvature are of wide interest in mathematical relativity and cosmology (see, e.g., \cite{Rendall1996}, \cite{Barbot2007}). Hereby, the mean curvature of the hypersurfaces (slices) of the foliation is assumed constant over the respective hypersurface but may vary over time. Twisted product manifolds are generalisations of warped product manifolds. Such manifolds of arbitrary signature were investigated in \cite{Ponge1993} and they have many applications in differential geometry for example also in relation to foliations (see, e.g., \cite{Rovenskii1998}). We investigate here CMC fooliations of twisted product spacetime (i.e., oriented Lorentzian manifolds) only in $2+1$-spacetime dimensions. The reason is a particularly simple formula for the mean curvature of the slices in this case (see equation (\ref{eqn:mc}) below).

\section{Preliminaries}
Let $(N,\gamma)$ be a $2$-dimensional Riemannian manifold, and $(\R,-\D t^2)$ the real line endowed with the negative definite standard metric. Then the twisted product $M=\R\times_{s}N$ is a $3$-dimensional Lorentzian manifold with the metric
\[
 g=-\D t^2+s^2\gamma
\] 
with $s\colon\R\times N\to\R^+$ being the twisting function. We denote points in the product space $M$ by $(t,x)\in\R\times M$. Hence the hypersurfaces $N_t:=\{t\}\times N\subset M$ for $t\in\R$ form a spacelike foliation of $M$. The vector field $V=\del_t$ is the lift of the canonical vector field $\del_t$ on $\R$ to $M=\R\times N$, which we denote be a slight abuse of notation with the same symbol. The vector field $V$ is everywhere perpendicular to the slices $N_t$ and has unit length, i.e., $g(V,V)=-1$. 

Let $h(t,x)=s^2(t,x)\gamma(x)$ be the induced metric on the slices $N_t$. We denote by $u=g(V,\cdot)$ the one-form metrically associated with $V$. Then the symmetric part of the projection of the covariant derivative of $u$ onto the slices $N_t$ can be decomposed into its trace-free part $\sigma$ and its trace $\theta$:
\[
 h(\textrm{sym}(\nabla u))=\frac{\theta}{2}h+\sigma.
\]
The trace $\theta$ is called expansion of $V$ and $\sigma$ is called shear of $V$ (cf.~\cite[Sec.~4.1]{Hawking1973}, \cite{Ehlers1993}). The expansion is given by
\[
 \theta=\textrm{div}(V)=2\frac{\dot{s}}{s},
\]
where the dot denotes the derivative with respect to $t$. Furthermore, we denote by $\dot{u}=g(\nabla_VV,\cdot)$ the acceleration of the vector field $V$. Thus the second fundamental form $K$ of the slices $N_t$ is given by
\[
 K=-\frac{1}{2}L_Vh=-\frac{\theta}{2}h-\sigma+\dot{u}\otimes u+u\otimes\dot{u}.
\]
Because of $h(V,\cdot)=0$, we get for the mean curvature $k$ of the slices $N_t$
\[
 k=\tr(K)=-\theta=-2\frac{\dot{s}}{s}.
\]
So obviously, if $M$ is a warped product, i.e., $s=s(t)$ (or if $s(t,x)=s_1(t)\cdot s_2(x)$, which can be regarded a warped product with Riemannian metric $s_2\gamma$ on $N$), then the slices $N_t$ form a CMC foliation of $M$. Hence, we can assume from now on that $M$ is a generic twisted product where the twisting function does not decompose as a product function.

Any change of the foliation of $M=\R\times N$ by $2$-dimensional spacelike slices is given in terms of a transformation
\[
 \psi\colon \R\times N\to\R\times N,\quad \psi\colon(t,x)\to (\tau(t,x),x) = (t+f(x),x)
\]
with $f$ being an arbitrary (smooth) function on $N$. Note that such a transformation can be regarded as a global gauge transformation in the trivial principal fiber bundle $\R\times N$, i.e., $\psi$ is an automorphism of the principal fibre bundle or a change of the trivialisation.

In the following, we will use indices $i,j\in\{1,2\}$ for local coordinates $x^i$ of $N$ and indices $a,b\in\{0,1,2\}$ for local coordinates $x^a$ of $M$. The transformation $\psi$ results in $\D t=\D\tau-\D f=\D\tau-f_{,i}\D x^i$. Hence the decomposition of the metric $g$ with respect to the new foliation is given by
\[
 g=-\D\tau^2+2\D f \D\tau+h-\D f\otimes \D f= -\D\tau^2 + 2f_{,i}\D x^i\D\tau + (h_{ij}-f_{,i}f_{,j})\D x^i\D x^j=
\]
\[
 =-\D\tau^2+2\D f\D \tau+s^2\gamma-\D f\otimes \D f= -\D\tau^2 + 2f_{,i}\D x^i\D\tau + (s^2\gamma_{ij}-f_{,i}f_{,j})\D x^i\D x^j.
\]
Obviously, it is necessary to have 
\begin{equation}\label{eqn:spacelike}
  \|\D f\|^h=\frac{\|\D f\|^{\gamma}}{s^2}=h^{ij}f_{,i}f_{,j}=\frac{\gamma^{ij}f_{,i}f_{,j}}{s^2}<1,
\end{equation}
so that the new slicing $\tilde{N}_{\tau}:=\{\tau\}\times N$ is spacelike, too.

Now we compute the normal vector field for $\tilde{N}_\tau$. This is given by
\[
 X=(X^{a})=(g^{ab}\D\tau_{b}).
\]
Using the shorthand notation $\del_i:=\frac{\del}{\del x_i}$ for a local basis of one-forms,  a tedious, but straightforward, calculation shows that 
\[
 g^{-1}=(\|\D f\|^h-1)\del_t^2 + 2h^{ij}f_{,j}\del_i + h^{ij}\del_i\del_j.
\]
Thus we have
\[
 X=(\|\D f\|^h-1)V + h^{ij}f_{,j}\del_i
\]
and now the expansion of $X$ and therefore the mean curvature $\tilde{k}$ of $\tilde{N}_{\tau}$ is given by
\begin{equation}\label{eqn:mc}
 -\tilde{k}=\textrm{div}(X)=\del_t(\|\D f\|^h) +(\|\D f\|^h-1)\theta + \triangle_hf=
\end{equation}
\[
 =\del_t\frac{\|\D f\|^{\gamma}}{s^2} + \frac{\|\D f\|^{\gamma}}{s^2}\cdot 2\frac{\dot{s}}{s}-2\frac{\dot{s}}{s}+ \triangle_hf=
\]
\[
 =\triangle_hf-2\frac{\dot{s}}{s},
\]
where $\triangle_h$ denotes the Laplace-Beltrami operator with respect to $h$ on the slices $N_t$. Note that the terms containing $\|\D f\|^{\gamma}$ do cancel out exactly because $\dim(N)=2$. As we have
\[
 \triangle_hf=\frac{1}{\sqrt{\det h}}\del_i\left(\sqrt{\det h}h^{ij}\del_jf\right)
\]
and $h^{ij}=s^{-2}\gamma^{ij}$ it follows that 
\[
 \triangle_hf= \frac{1}{s^2}\triangle_{\gamma}f.
\]

Note that $t$ is a time function for the spacetime $(M,g)$ and so is $\tau$. Hence $(M,g)$ is stably causal (see, e.g., \cite{Minguzzi2008} for an overview of causality theory for Lorentzian manifolds). The preimages of $t$ (resp.~$\tau$) are the spacelike slices $N_t$ (resp.~$N_\tau$). If the preimages of a time function are hypersurfaces of constant mean curvature, it is called a CMC time function (cf., e.g., \cite{Andersson2012}).

Also note that due to $\tau=t+f(x)$, we have $\del_t=\del_\tau$, i.e., the dot denotes derivatives by $t$ and $\tau$. 

\section{Main Result}

Now we are able to state the following

\begin{Thm}
A $2+1$-dimensional twisted product spacetime
\[
 (M,g)=(\R\times N,-\D t^2+s^2\gamma)
\]
with $s\colon M\to \R^+$ a function and $\gamma$ a Riemannian metric on $N$, has a foliation by spacelike surfaces of constant mean curvature if there are functions $\alpha\colon\R\to\R^+$, $\beta\colon N\to\R$, $\xi\colon N\to\R$ and $f\colon N\to\R^+$, such that the following three conditions hold 
\begin{itemize}
 \item[(i)] $s^2(\tau,x)=\frac{\alpha(\tau)\beta(x)+\xi(x)}{\dot{\alpha}(\tau)}$, 
 \item[(ii)] $\triangle_{\gamma}f=\beta$ and
 \item[(iii)] $\|\D f\|^{\gamma}<s^2$,
\end{itemize}
with $\tau = t + f(x)$. Moreover, in this case, $\tau$ is a CMC time function.
\end{Thm}

\begin{proof}
Using (i) and (ii), we can calculate the mean curvature $\tilde{k}$ to be
\[
 -\tilde{k}=\frac{1}{s^2}\triangle_{\gamma}f-2\frac{\dot{s}}{s}=\frac{\dot{\alpha}}{\alpha\beta+\xi}\beta-\frac{1}{\frac{\alpha\beta+\xi}{\dot{\alpha}}}\left(\beta-\frac{\ddot{\alpha}}{\dot{\alpha}^2}(\alpha\beta+\xi)\right)=\frac{\ddot{\alpha}}{\dot{\alpha}}(\tau),
\]
which is constant on any slice $\tilde{N}_{\tau}$. Condition (iii) assures the CMC foliation to be spacelike, as can be seen by comparision to (\ref{eqn:spacelike}), hence $\tau$ is a time function. And $\tau$ is even a CMC time function, because its level sets $\tilde{N}_{\tau}$ are all spacelike CMC surfaces. 
\end{proof}


\end{document}